\documentclass[11pt]{article}

\usepackage{amsmath}
\usepackage{amsfonts}
\usepackage{amssymb}
\usepackage[english]{babel}
\usepackage{graphicx}
\usepackage{amsthm}
\usepackage{xcolor}
\usepackage{authblk}
\usepackage{url}

\newtheorem{thm}{Theorem}[section]
\newtheorem{cor}[thm]{Corollary}
\newtheorem{lemma}[thm]{Lemma}

\newtheorem{ex}{Example}[section]

\theoremstyle{definition}

\theoremstyle{remark}

\newcommand{\pc}{\rho}

\newcommand{\B}{\mathcal{B}}
\newcommand{\STS}{\mathcal{S}}
\newcommand{\zed}{\mathbb{Z}}

\begin{document}
\date{\today}
\title{On partial parallel classes in partial Steiner triple systems}
\author{Douglas R.\ Stinson\thanks{D.R.\ Stinson's research is supported by  NSERC discovery grant RGPIN-03882.
}\\David R.\ Cheriton School of Computer Science\\University of Waterloo\\Waterloo, Ontario, N2L 3G1, Canada}

\maketitle

\begin{abstract}
For an integer $\pc$ such that $1 \leq \pc \leq v/3$, define $\beta(\pc,v)$ to be the maximum number of blocks in any partial Steiner triple system on $v$ points in which the maximum partial parallel class has size $\pc$. We  obtain lower bounds on 
$\beta(\pc,v)$ by giving explicit constructions, and upper bounds on $\beta(\pc,v)$ result from counting arguments. We show that
$\beta(\pc,v) \in \Theta (v)$ if $\pc$ is a constant, and $\beta(\pc,v) \in \Theta (v^2)$ if $\pc = v/c$, where $c$ 
is a constant. When $\pc$ is a constant, our upper and lower bounds on $\beta(\pc,v)$ differ by a constant that 
depends on $\pc$. Finally, we apply our results on $\beta(\pc,v)$ to obtain infinite classes of sequenceable partial Steiner triple systems.
\end{abstract}

\section{Introduction}

Suppose $v$ is a positive integer. A \emph{partial Steiner triple system of order $v$}, denoted PSTS$(v)$, is a pair $\STS = (X, \B)$, where $X$ is a set of $v$ \emph{points} and $\B$ is a set of
 3-subsets of $X$, called \emph{blocks}, such that every pair of points occurs in at most one block. The number of blocks in 
 $\B$ is usually denoted by $b$. 
 If every pair of points occurs in exactly one block, 
 the PSTS$(v)$ is a  \emph{Steiner triple system of order $v$}, denoted STS$(v)$.
 It is clear that a PSTS$(v)$ is an STS$(v)$ if and only if $b = v(v-1)/6$. An STS$(v)$ exists if and only if $v \equiv 1,3 \bmod 6$. 
Colbourn and Rosa \cite{CR} is the definitive reference for 
Steiner triple systems and related structures. 

Let $D(v)$ denote the maximum number of blocks in a PSTS$(v)$.
The following theorem gives the values of $D(v)$ for all positive integers $v$.
\begin{thm}
\label{packing.thm}
Suppose $v \geq 3$. Then
\[ D(v) =
\begin{cases}
\left\lfloor  \frac{v}{3} \left\lfloor  \frac{v-1}{2} \right\rfloor \right\rfloor & \text{if $v\not\equiv 5 \bmod 6$}\\
\left\lfloor  \frac{v}{3} \left\lfloor  \frac{v-1}{2} \right\rfloor \right\rfloor - 1 & \text{if $v \equiv 5 \bmod 6$.}
\end{cases}
\]
\end{thm}

Suppose $\STS = (X, \B)$ is a PSTS$(v)$. For any point $x \in X$, the \emph{degree of $x$}, 
denoted $d_x$, is the number of blocks in $\B$ that contain $x$. 
A PSTS$(v)$ is \emph{$d$-regular} if $d_x = d$ for all $x \in X$.
It is clear that $d_x \leq (v-1)/2$ for all $x$. Also, an STS$(v)$ is a $d$-regular PSTS$(v)$ with $d = (v-1)/2$.

Suppose $\STS = (X, \B)$ is a PSTS$(v)$. A \emph{parallel class} in $\STS$ is a set of $v/3$ disjoint blocks in $\B$.
A \emph{partial parallel class} (or \emph{PPC}) in $\STS$  is any set of disjoint blocks in $\B$.
The \emph{size} of a PPC is the number of blocks in the PPC. A PPC in $\STS$ of size $\pc$  is \emph{maximum} if there does not exist a PPC in $\STS$ of size $\pc +1$.



There has been considerable study concerning the sizes of maximum partial parallel classes in STS$(v)$. It seems to be quite difficult to find STS$(v)$ with $v\equiv 3 \bmod 6$ that do not have parallel classes. It was only in 2015 that the first infinite classes were found, by Bryant and Horsley \cite{BH}. For STS$(v)$ with $v\equiv 1 \bmod 6$ that do not have partial parallel classes of size $(v-1)/3$, there are two infinite classes known, one of which was found by Wilson (see \cite{RC})
and one discovered by Bryant and Horsley \cite{BH2}.

There are also lower bounds on the sizes of maximum partial parallel classes in STS$(v)$.  The first nontrivial bound is due to Lindner and Phelps \cite{LP}, who proved in 1978 that every STS$(v)$ with $v>27$ has a partial parallel class of size at least $\frac{v-1}{4}$. Improvements have been found by Woolbright \cite{Wo} and
Brouwer \cite{Br}. The result proven in \cite{Br} is that every  STS$(v)$ has a PPC of size at least
$\frac{v}{3} - \frac{5}{3}v^{2/3}$.  In 1997, Alon, Kim and Spencer \cite{AKS} used probabilistic methods to prove that any STS$(v)$ contains a PPC of size at least $\frac{v}{3} - O(v^{1/2}(\ln v)^{3/2})$. Also, a recent preprint by 
Keevash, Pokrovskiy, Sudakov and Yepremyan \cite{KPSY} improves this lower bound to 
$\frac{v}{3} - O\left(\frac{\log v}{ \log\log v}\right)$.

We should also mention Brouwer's conjecture, made in \cite{Br}, that any STS$(v)$ has a PPC of size at least $(v-4)/3$.

The problem of determining lower bounds on the the sizes of maximum partial parallel classes 
in partial Steiner triple systems has received less study. It was shown in 
\cite{AKS} that a $d$-regular PSTS$(v)$ has a PPC of size at least 
$\frac{v}{3} - O\left( \frac{v \, (\ln d)^{3/2}} { d^{1/2}}\right)$.

In this paper, we address the following problem. For an integer $\pc$ such that $1 \leq \pc \leq v/3$, define $\beta(\pc,v)$ to be the maximum number of blocks in any PSTS$(v)$ in which the maximum partial parallel class has size $\pc$. We  obtain lower bounds on 
$\beta(\pc,v)$ by giving explicit constructions, and upper bounds result from counting arguments.

Before continuing, we observe that, when $v \equiv 1,3 \bmod 6$, the values  $\beta(\lfloor \frac{v}{3} \rfloor,v)$ 
follow easily from known results concerning STS$(v)$.

\begin{thm} 
\label{max.thm}
For all $v \equiv 1,3 \bmod 6$, $v \neq 7$, it holds that $\beta(\lfloor \frac{v}{3} \rfloor,v) = v(v-1)/6$.
Also,  $\beta(2,7) = 5$.
\end{thm}

\begin{proof}
For all $v \equiv 3 \bmod 6$, it is known that there exists an STS$(v)$ containing a parallel class.
For all $v \equiv 1 \bmod 6$, $v\neq 7$, there exists an STS$(v)$ containing a PPC of size $(v-1)/3$.
For $v = 7$, the PSTS$(7)$ consisting of blocks $123,456,147,257,367$ has a maximum PPC of size $2$, and there
is no PSTS$(7)$ having six blocks that has a maximum PPC of size $2$.
\end{proof}

As mentioned above, there are STS$(v)$ in which the maximum partial parallel class has size less than 
$\lfloor v/3\rfloor  - 1$. In this situation, we would have $\beta(\pc,v) = v(v-1)/6$, where 
$\pc$ is the size of the maximum PPC in the given STS$(v)$. For example, there is an STS$(15)$ having a maximum PPC of size $4$, as well as  STS$(21)$ having a maximum PPC of size $6$. Therefore, $\beta(4,15) = 35$ and $\beta(6,21) = 70$.

We summarize the main contributions of this paper.
In Section \ref{construction.sec}, we present constructions for partial Steiner triple systems that have maximum PPCs of a prespecified size, thus obtaining lower bounds on $\beta(\pc,v)$. In Section \ref{bounds.sec}, we  prove an upper bound on 
$\beta(\pc,v)$ using a counting argument. In Section \ref{analysis.sec}, we show that
$\beta(\pc,v) \in \Theta (v)$ if $\pc$ is a constant, and $\beta(\pc,v) \in \Theta (v^2)$ if $\pc = v/c$, where $c$ 
is a constant. When $\pc$ is a constant, our upper and lower bounds on $\beta(\pc,v)$ differ by a constant that 
depends on $\pc$.
In Section \ref{sequencing.sec}, we apply our results to the problem of finding sequenceable PSTS$(v)$. In particular, the constructions we describe in Section \ref{construction.sec} provide infinite classes of sequenceable PSTS$(v)$.
Finally, Section \ref{summary.sec} is a brief summary.

\section{PSTS with Small Maximum PPCs}
\label{construction.sec}

We present a construction for partial Steiner triple systems that have maximum PPCs of a prespecified size. This construction utilizes Room squares, which we define now. Suppose that $\ell$ is even and let $T$ be a set of size $\ell$. A 
\emph{Room square of side $\ell-1$ on symbol set $T$} is an $\ell$ by $\ell$ array $R$ that satisfies the following conditions:
\begin{enumerate}
\item every cell of $R$ either is empty or contains an edge from the complete graph $K_{\ell}$ on symbol set $T$,
\item every edge of $K_{\ell}$ occurs in exactly one cell of $R$,
\item the filled cells in every row of $R$ comprise a one-factor of $K_{\ell}$, and 
\item the filled cells in every column of $R$ comprise a one-factor of $K_{\ell}$.
\end{enumerate}
It is well-known that a Room square of side $\ell-1$ exists if and only if $\ell \geq 2$ 
is even and $\ell \neq 4$ or $6$; see \cite{MW}.

\begin{thm}
\label{construction1}
Let $\ell \geq 2\pc$ be even and assume $(\ell,\pc) \neq (4,2)$. 
Then there exists a PSTS$(\pc + \ell)$ having $\pc \ell/2$ blocks, in which the 
largest partial parallel class has size $\pc$. 
\end{thm}

\begin{proof}
Let $S = \{s_i : 1 \leq i \leq \pc\}$ and let 
$T$ be a set of size $\ell$ that is disjoint from $S$.
Let $F_1, \dots , F_\pc$ be $\pc$ edge-disjoint one-factors of the complete graph $K_{\ell}$ 
on vertex set $T$. For $1 \leq j \leq \pc$, suppose $e_j \in F_j$ is chosen such that the $\pc$ edges $e_1, \dots ,e_\pc$ are independent. 

We can easily do this  if $\ell \geq 8$ is even: start with a Room square of side $\ell-1$ on symbol set $T$. 
The $\pc$ rows that contain a pair in the first column are the one-factors $F_1, \dots , F_{\pc}$. The edges $e_1, \dots ,e_\pc$
are the pairs in the first column of these rows.

We consider the cases $\ell = 2,4$ or $6$ separately. If $\pc=1$, there is nothing to prove, so we can assume
$2 \leq \pc \leq \ell/2$. 
Suppose $\ell = 4$ and $\pc=2$. It is easily seen that there do not exist two disjoint one-factors of $K_4$ along with two independent edges, one from each one-factor. Next, suppose $\ell = 6$ and $\pc=3$. Here we can use the three one-factors
$\{ 03, 14, 25\}$, $\{01,23,45\}$ and $\{12,34,50\}$ along with the three independent edges $03, 45, 12$. Finally, for 
$\ell = 6$ and $\pc=2$, we use  any two of these three one-factors along with the corresponding independent edges.

Now, adjoin $s_j$ to every edge in $F_j$, for $1 \leq j \leq \pc$. The result is a PSTS$(\pc+ \ell)$ containing $\pc \ell/2$ blocks.
It is easy to see that the largest partial parallel class has size $\pc$. First, the $\pc$ blocks containing the edges $e_1, \dots ,e_\pc$ are disjoint, so we have a partial parallel class of size $\pc$. Second, every block contains a point from $S$, so there does not exist a partial parallel class of size $\pc+1$.
\end{proof}

Here is a slight improvement of Theorem \ref{construction1}.

\begin{thm}
\label{construction2}
Let $\ell \geq  2\pc$ be even and assume $(\ell,\pc) \neq (4,2)$. 
Then there exists a PSTS$(\pc + \ell)$ having $\frac{\pc \ell}{2} + D(\pc)$ blocks, in which the 
largest partial parallel class has size $\pc$.
\end{thm}

\begin{proof}
Use Theorem \ref{construction1} to construct a PSTS$(\pc + \ell)$ having $\frac{\pc \ell}{2}$ blocks, such that the 
largest partial parallel class has size $\pc$.  We add $D(p)$ additional blocks that comprise a 
maximum PSTS$(\pc)$ on the points in the set $S$. The new PSTS cannot have a partial parallel class of size exceeding $\pc$ because every block still meets $S$ in at least one point.
\end{proof}

\begin{ex} We illustrate Theorem \ref{construction2} with $\pc = 3$ and $\ell= 8$. We start with a Room square of side $7$:
\[
\begin{array}{|c|c|c|c|c|c|c|c}
\cline{1-7}
\textcolor{red}{70} &    &    & 64 &    & 32 & 51 & \quad  F_1 \\  \cline{1-7}
\textcolor{red}{26} & 71 &    &    & 05 &    & 43 & \quad  F_2 \\  \cline{1-7}
\textcolor{red}{54} & 30 & 72 &    &    & 16 &    & \quad  F_3 \\  \cline{1-7} 
   & 65 & 41 & 73 &    &    & 20 &     \\  \cline{1-7} 
31 &    & 06 & 52 & 74 &    &    &  \\  \cline{1-7} 
   & 42 &    & 10 & 63 & 75 &    &     \\  \cline{1-7} 
   &    & 53 &    & 21 & 04 & 76 &     \\  \cline{1-7}  
\end{array}
\]
The one-factors $F_1, F_2 , F_3$ are indicated, and the edges $e_1, e_2 , e_3$ appear in red.
From these three one-factors, we obtain $12$ blocks of size three on the points
in $S \cup T$, where  $S = \{s_1, s_2,s_3\}$ and $T = \{0, \dots , 7\}$. We can adjoin 
$D(3) = 1$ additional block on $S$, obtaining a PSTS$(11)$ having  $b = 13$ blocks:
\[ 
\begin{array}{llll}
\textcolor{red}{\{s_1,7,0\}} & \{s_1,6,4\} & \{s_1,3,2\} & \{s_1,5,1\} \\ 
\textcolor{red}{\{s_2,2,6\}} & \{s_2,7,1\} & \{s_2,0,5\} & \{s_2,4,3\} \\ 
\textcolor{red}{\{s_3,5,4\}} & \{s_3,3,0\} & \{s_3,7,2\} & \{s_3,1,6\} \\
\{s_1,s_2,s_3\} 
\end{array}
\]
This PSTS$(11)$ has a maximum parallel class of size three, which is indicated in red.
$\blacksquare$
\end{ex}

Theorems \ref{construction1} and
\ref{construction2}  can only be applied when $v - \pc$ is even. We present a simple variation that accommodates 
odd values of $v -\pc$.

\begin{thm}
\label{construction3}
Let $\ell >  2\pc$ be even. 
Then there exists a PSTS$(\pc + \ell - 1)$ having $\frac{\pc \ell}{2} + D(\pc) - \pc$ blocks, in which the 
largest partial parallel class has size $\pc$.
\end{thm}

\begin{proof}
Use Theorem \ref{construction2} to construct a PSTS$(\pc + \ell)$ having $\frac{\pc \ell}{2} + D(\pc)$ blocks, such that the 
largest partial parallel class has size $\pc$. Since $2\pc < \ell$, there exists a point $x \in T \setminus ( \bigcup _{i=1}^{\pc} e_i )$. Delete the $\pc$ blocks that contain $x$ and note that none of these blocks are contained in the PPC of size $\pc$.
\end{proof} 

Combining Theorems \ref{construction2} and \ref{construction3}, we have the following lower bounds on $\beta(\pc,v)$.

\begin{thm}
\label{const.thm}
\quad 
\begin{enumerate}
\item If $v \geq  3\pc$, $v -\pc$ is even and  $(v,\pc) \neq (6,2)$, then 
\[ \beta(\pc,v) \geq \frac{\pc (v-\pc)}{2} + D(\pc).\]
\item If $v >  3\pc$ and $v -\pc$ is odd, then  
\[\beta(\pc,v) \geq \frac{\pc (v - \pc-1)}{2} + D(\pc).\]
\end{enumerate}
\end{thm}

\section{An Upper Bound on $\beta(\pc,v)$}
\label{bounds.sec}

In this section, we prove an upper bound on $\beta(\pc,v)$ using a simple counting argument.
The proof of this bound depends on a very useful observation due to  Lindner and Phelps \cite{LP}. 
Suppose $\STS = (X, \B)$ is a PSTS$(v)$ in which the maximum
maximum partial parallel class has size  $\pc$; hence $v \geq 3\pc$.
Let $\mathcal{P} = \{B_1, \dots , B_\pc\}$ be a set of $\pc$ disjoint blocks 
and let $P = \bigcup_{i=1}^{\pc}  B_i$. 

For $x\in P$, let $T_x$ denote the set of blocks 
that contain $x$ and two points in $X \setminus P$.  Then $T_x$ is a partial one-factor of $X \setminus P$.
Define $t_x = |T_x|$; then $t_x \leq (v-3\pc)/2$.

Suppose $t_x \geq 3$ and $t_y \geq 1$ for some pair $\{x,y\} \in B_i$. Then, as noted in \cite{LP}, 
it is clear that there exists a block $A \in T_x$ such that
$A \cap B = \varnothing$, where $B$ is the unique block in $T_y$. Then 
\[ \mathcal{P} \setminus \{B_i\} \cup \{ A,B\}\] 
is a partial parallel class of size $\pc+1$, which is a contradiction.

Denote $X_0 = \{x \in P: t_x >0 \}$. In view of the above discussion, 
the proof of the following lemma is straightforward.

\begin{lemma}
\cite{LP}
\label{tx.lem}
For any block $B_i  = \{x,y,z\} \in \mathcal{P}$, one of the following two conditions holds:
\begin{enumerate}
\item $|B_i \cap X_0| \geq 2$ and $t_x+t_y+t_z \leq 6$.
\item $|B_i \cap X_0| \leq 1$ and $t_x+t_y+t_z \leq (v-3\pc)/2$.
\end{enumerate}
\end{lemma}

  \begin{thm}
 \label{rho.thm}
 Suppose $v\geq 3\pc$ and $\STS = (X,\B)$ is a  PSTS$(v)$ having a maximum partial parallel class of size $\pc$. 
Then the number of blocks, $b$, satisfies the following inequality:
\begin{equation}
\label{upperbound.eq}
b \leq \pc \left( \frac{9\pc  - 7}{2} + \max \left\{6, \left\lfloor \frac{v-3\pc}{2} \right\rfloor \right\}
 \right).
 \end{equation}
 \end{thm}
 
 \begin{proof}
Let $\mathcal{P} = \{B_1, \dots , B_\pc\}$ be a set of $\pc$ disjoint blocks in $\B$
and let $P = \bigcup_{i=1}^{\pc}  B_i$.  
The $\pc$ blocks in $\mathcal{P}$ each include three points from $P$. These blocks cover 
$3\pc$ pairs of points from $P$. The number of pairs of points from $P$ that do not occur in the blocks in $\mathcal{P}$
is $\binom{3\pc}{2} - 3\pc$.
Therefore, there are at most 
\[ \pc + \binom{3\pc}{2} - 3\pc = \binom{3\pc}{2} - 2\pc\] blocks in $\B$ that contain at least two points in $P$.  
All remaining blocks in $\B$ contain exactly one point in $P$.  It therefore follows immediately from Lemma \ref{tx.lem} that 
the number of blocks in $\B$ is at most
 \begin{align*}
&\binom{3\pc}{2} - 2\pc + \pc \times \max \left\{6, \left\lfloor \frac{v-3\pc}{2} \right\rfloor \right\}
 \\
 = &\pc \left( \frac{9\pc  - 7}{2} + \max \left\{6, \left\lfloor \frac{v-3\pc}{2} \right\rfloor \right\}
 \right) .
\end{align*}
\end{proof}

We have the following corollary of Theorem \ref{rho.thm}.

  \begin{cor}
 \label{rho.cor}
 Suppose $v \geq 3 \pc + 12$. Then
 \begin{equation}
\label{simpleupperbound.eq}
\beta(\pc,v) \leq \pc \left( \frac{6\pc +v - 7}{2}  \right) .
\end{equation}
\end{cor}

\begin{proof}  
For $v \geq 3 \pc + 12$, we have
\[ \max \left\{6, \left\lfloor \frac{v-3\pc}{2} \right\rfloor \right\} =  \left\lfloor \frac{v-3\pc}{2} \right\rfloor 
\leq  \frac{v-3\pc}{2}.\]
It follows from  Theorem \ref{rho.thm} that 
\[ \beta(\pc,v) \leq
\pc \left( \frac{9\pc  - 7}{2} + \frac{v-3\pc}{2} \right) =  \pc \left( \frac{6\pc +v - 7}{2}  \right)
.
 \]
\end{proof} 
 
Define
\[f(\pc) = \pc \left( \frac{6\pc +v - 7}{2}  \right) .\]
For fixed $v$, $f(\pc)$ is an increasing function of $\pc$. 
Now, suppose we take $\pc =\frac{v+3}{6}$.  We compute
\begin{align*}
f\left(\frac{v+3}{6}\right) &= \left(\frac{v+3}{6}\right) \left( \frac{v+3+v-7}{2}\right)\\
& = \left(\frac{v+3}{6}\right) \left( v-2 \right)\\
& = \frac{v^2+v-6}{6}\\
& \geq \frac{v^2-v}{6},
\end{align*}
since $v \geq 3$.
Therefore, the bound (\ref{simpleupperbound.eq}) is useful only when  $\pc < (v+3)/6$.
For $\pc \geq (v+3)/6$, we just have the trivial upper bound 
\[\beta(\pc,v) \leq \frac{v(v-1)}{6}.\]

 \section{Analyzing the Bounds}
 \label{analysis.sec}

First, we observe that we can compute the exact value of $\beta(1,v)$ for any $v \geq 3$.

\begin{thm}
\label{beta1.thm}
Suppose $v \geq 3$. Then
\[ 
\beta (1,v) =
\begin{cases}
1 & \text{if $v= 3,4$}\\
2 & \text{if $v= 5$}\\
4 & \text{if $v= 6$}\\
7 & \text{if $7 \leq v \leq 14$}\\
\lfloor \frac{v-1}{2} \rfloor & \text{if $v \geq 15$.}
\end{cases}
\]
\end{thm}

\begin{proof}
For any $v \geq 3$, we have 
\begin{equation}
\label{rho1.eq}
\beta(1,v) \leq \max \left\{7, \left\lfloor \frac{v-1}{2} \right\rfloor \right\}
\end{equation} from (\ref{upperbound.eq}).
When $v \geq 15$, (\ref{rho1.eq}) simplifies to $\beta(1,v) \leq \lfloor \frac{v-1}{2} \rfloor$.
Applying Theorem \ref{const.thm}, we see that $\beta(1,v) \geq \lfloor \frac{v-1}{2} \rfloor$ for all $v\geq 3$.
Therefore $\beta(1,v) = \lfloor \frac{v-1}{2} \rfloor$ for $v \geq 15$.

For $v \leq 14$, (\ref{rho1.eq}) yields $\beta(1,v) \leq 7$.
For $7 \leq v \leq 14$, if we take the seven blocks of an STS$(7)$, we see
that $\beta(1,v) \geq 7$. Hence $\beta(1,v) = 7$ for $7 \leq v \leq 14$.

The cases  $3 \leq v \leq 6$ can be analyzed separately. The optimal PSTS$(v)$ with $\pc =1$ are as follows:
$\{123\}$ for $v = 3,4$; $\{123,145\}$ for $v = 5$; and $\{123,145,246,356\}$ for $v=6$.
\end{proof}

We have the following bounds on $\beta(2,v)$ from Theorem \ref{const.thm} and Corollary \ref{rho.cor}.

\begin{thm}
\label{beta2.thm}
For $v \geq 18$, it holds that
\[ v-3 \leq \beta(2,v) \leq v+1.\]
\end{thm}

In a similar way, we can get bounds on $\beta(3,v)$.

\begin{thm}
\label{beta3.thm}
For $v \geq 21$, it holds that
\[ \frac{3v-10}{2} \leq \beta(3,v) \leq \frac{3v+33}{2}.\]
\end{thm}

Suppose we fix a value of $\pc$ and suppose $v \geq  3\pc + 12$.
From Theorem \ref{const.thm} and Corollary \ref{rho.cor}, we have 
\[ \beta(\pc,v) \geq \beta_L(\pc,v) \triangleq \frac{\pc (v - \pc-1)}{2} + D(\pc)\] and 
\[ \beta(\pc,v) \leq \beta_U(\pc,v) \triangleq \pc \left( \frac{6\pc +v - 7}{2}  \right) .\]
Therefore, since $D(\pc) \in \Theta(\pc^2)$, it immediately follows that $\beta(\pc,v) \in \Theta(v)$ if $\pc$ is a constant.

The upper and lower bounds in Theorem \ref{beta2.thm} and Theorem \ref{beta3.thm} differ by a constant.
This in fact occurs for any fixed value of $\pc$, as we show now. 
We can bound the difference $\beta_U(\pc,v) -  \beta_L(\pc,v)$ as follows:
\begin{align*}
\beta_U(\pc,v) - \beta_L(\pc,v) &\leq \pc \left( \frac{6\pc +v - 7}{2}  \right) - \left( \frac{\pc (v - \pc-1)}{2} + D(\pc) \right)\\
& = \frac{7 \pc^2 - 6\pc}{2} - D(\pc).
\end{align*}
Now, from Theorem \ref{packing.thm}, it can be verified that
\[ D(v) \geq \frac{v(v-2) - 2}{6}.\]
Hence, we have 
\begin{align}
\nonumber \beta_U(\pc,v) - \beta_L(\pc,v) &\leq \frac{7 \pc^2 - 6\pc}{2} - \frac{\pc(\pc-2) - 2}{6}\\
\label{diff.eq} & = \frac{10 \pc^2 - 8\pc +1}{3}.
\end{align}
Therefore,  the difference between our upper and lower bounds is at most a constant
(which depends on $\pc$).

On the other hand, if  $\pc = v/c$  where $c \geq 3$ is a constant, then 
it follows immediately from Theorem \ref{const.thm} and Corollary \ref{rho.cor} that 
$\beta\left(\frac{v}{c},v\right) \in \Theta(v^2)$.  

Finally, it is of course easy to use Theorems \ref{const.thm} and \ref{rho.thm} to 
find upper and lower bounds on $\beta(\pc,v)$ for any specified value of $v$. 
We illustrate by tabulating these bounds for $v= 27$ in Table \ref{tab1}. 
Theorem \ref{rho.thm} gives a nontrivial upper bound on $\beta(\pc,27)$ for $1 \leq \pc \leq 4$. For
$5 \leq \pc \leq 8$, we only have the trivial upper bound $\beta(\pc,27) \leq 27 \times 26 / 6 = 117$.
However, $\beta(9,27) = 117$ from Theorem \ref{max.thm}. 

We can also show that $\beta(8,27) = 117$. The construction given by Bryant and Horsley in \cite{BH} yields an STS$(27)$ that does not contain a parallel class. It turns out to be possible to find a PPC of size $8$ in this STS$(27)$.  We do not give a complete description of how the design is constructed. However, we note that the point set of the design is 
$(\zed_5 \times \zed_5) \cup \{\infty_1,\infty_2\}$ and the blocks include all subsets $\{x,y,z\}$ 
of three distinct elements in
$\zed_5 \times \zed_5$ such that $x + y + z = (0,0)$.
The following eight disjoint triples are therefore blocks in this STS$(27)$:
\[
\begin{array}{ll}
\{10,11,34\}\quad 	&	\quad \{01,31,23\}\\
\{20,22,13\}\quad 	&	\quad \{02,12,41\}\\
\{30,33,42\}\quad 	&	\quad \{03,14,43\}\\
\{40,44,21\}\quad 	&	\quad \{04,32,24\}.
\end{array}
\]
These  blocks comprise a PPC of size eight.

\begin{table}
\caption{Upper and lower bounds on $\beta(\pc,27)$}
\label{tab1}
\[
\begin{array}{c|c|c|c}
\pc & D(\pc) & \text{lower bound} & \text{upper bound} \\ \hline
1 & 0 & 13 & 13 \\
2 & 0 & 24 & 31\\
3 & 1 & 37 & 57\\
4 & 1 & 45 & 86\\
5 & 2 & 57 & 117\\
6 & 4 & 64 & 117\\
7 & 7 & 77 & 117\\
8 & 8& 117 & 117 \\
9 & 12 & 117 & 117 
\end{array}
\]
\end{table}

 \section{Sequencings of PSTS$(v)$}
 \label{sequencing.sec}
 
 
A PSTS$(v)$, say $\STS= (X,\B)$, is \emph{sequenceable}
if there a permutation $\pi$ of $X$ such that no $3t$ consecutive points in $\pi$ 
is the union of $t$ blocks in $\B$, for all $t$ such that $1 \leq t \leq  \lfloor \frac{v}{3}\rfloor$.
The question of determining which PSTS$(v)$ are sequenceable was introduced by Alspach \cite{Als}. 
This problem has been further studied in \cite{AKP,BE,KS}. It has been shown that a PSTS$(v)$ is 
sequenceable if the size of the 
maximum PPC satisfies certain conditions.  The known results are summarized in the following theorem.
 
 \begin{thm}
 \label{sequenceable.thm}
 Suppose $\STS= (X,\B)$ is a PSTS$(v)$.  Then $\STS$ is sequenceable if any of the following 
 conditions is satisfied:
 \begin{enumerate}
 \item The size of a maximum partial parallel class in $\STS$ is at most three.
 \item The size of a maximum partial parallel class in $\STS$ is $\pc$ and $v \geq 15\pc - 5$.
 \item The size of a maximum partial parallel class in $\STS$ is $\pc$ and $v \geq 9\pc + 22 \pc^{2/3} +10$.
 \end{enumerate}
 \end{thm}
 
 \begin{proof}
 1.\ and 2.\ are shown in Alspach,  Kreher and  Pastine \cite{AKP} and 3.\ is proven in Blackburn and Etzion \cite{BE}.
 \end{proof}
 
On the other hand, it is known that there is a nonsequenceable 
STS$(v)$ for all $v \equiv 1 \bmod 6$, $v > 7$; see Kreher and Stinson \cite{KS}. 
There are currently no known examples of nonsequenceable 
STS$(v)$ for $v \equiv 3 \bmod 6$.

Theorems \ref{construction1}, \ref{construction2} and \ref{construction3}  can be used to construct examples of PSTS$(v)$ 
that satisfy the hypotheses of Theorem \ref{sequenceable.thm}. Hence, 
the resulting PSTS$(v)$ are necessarily sequenceable. 
In this way, we can obtain sequenceable PSTS$(v)$ with $b \in \Theta(v)$ for $\pc \leq 3$, as well as
sequenceable PSTS$(v)$ with $b \in \Theta(v^2)$ for $\pc  \in \Theta(v)$.
 
On the other hand, Corollary \ref{rho.cor} shows that a PSTS$(v)$ with a sufficiently large number of blocks cannot
satisfy the hypotheses of Theorem \ref{sequenceable.thm}. (Of course this not mean that such a PSTS$(v)$ is nonsequenceable.)

To illustrate, suppose  $\pc = 3$ and $v \geq 21$. Then Theorem \ref{beta3.thm}  asserts that 
$\beta(3,v) \leq \frac{3v+33}{2}$. Hence, a PSTS$(v)$ with $v \geq 21$ having a maximum PPC of size $3$ has at most  $\frac{3v+33}{2}$ blocks. Therefore, Theorem \ref{sequenceable.thm} cannot be applied with $\pc = 3$ if the number of blocks in a PSTS$(v)$ is greater than this number. For example, if $v = 21$, we conclude that a PSTS$(21)$ with at least $49$ blocks must have a PPC of size at least $4$ and therefore the hypotheses of Theorem \ref{sequenceable.thm} are not satisfied for such a PSTS. 
On the other hand, from Theorem \ref{construction2}, there is a PSTS$(21)$ with $b = 28$ blocks in which the maximum PPC has size $3$. From Theorem \ref{sequenceable.thm}, this PSTS$(21)$ is sequenceable.

\section{Summary}
\label{summary.sec}

It would be of interest to obtain tighter bounds on the values $\beta(\pc,v)$. 
However, our upper and lower bounds on $\beta(\pc,v)$ are already relatively close. It is not difficult to see why this is the case.
Our main construction, Theorem \ref{construction2}, produces an STS$(v)$ in which condition 2.\ of Lemma \ref{tx.lem} is met with equality, for all $\pc$ blocks in $\mathcal{P}$. Therefore, our upper and lower bounds differ only because of blocks that contain two or three points from $P$. Our upper bound supposes that all such blocks 
(other than the blocks in $\mathcal{P}$) contain two points from $P$. On the other hand, our construction only includes blocks that contain three points from $P$. If the gap between the upper and lower bounds is to be decreased, this would be the area of focus.

Aside from trying to tighten the upper and lower bounds in general, it would be of particular interest to find nontrivial upper bounds on $\beta(\pc,v)$ when $\pc \geq (v+3)/6$.

\section*{Acknowledgements}
I would like to thank Simon Blackburn and Daniel Horsley for helpful comments.

\end{document}